\documentclass[11pt]{article}
\usepackage{amssymb}
\usepackage{amsmath}
\usepackage{graphicx}

\newtheorem{thm}{Theorem}
\newtheorem{cor}[thm]{Corollary}

\newtheorem{lemma}[thm]{Lemma}

\newtheorem{proposition}[thm]{Proposition}
\newtheorem{claim}[thm]{Claim}
\newtheorem{rem}[thm]{Remark}
\newenvironment{proof}{\noindent\bf{Proof.}\rm}{\hfill$\blacksquare$\bigskip}

\newcommand{\etal}{{\em et al.}} 

\begin{document}

\title{Random walks with the minimum degree local rule have $O(n^2)$ cover time}

\author{
	Roee David\thanks{
		Department of Computer Science and Applied Mathematics, Weizmann Institute of Science, Rehovot 76100,
		Israel. E-mail: {\tt roee.david@weizmann.ac.il}.}
	\and Uriel Feige\thanks{
		Department of Computer Science and Applied Mathematics, Weizmann Institute of Science, Rehovot 76100,
		Israel. E-mail: {\tt uriel.feige@weizmann.ac.il}.}
	}

\maketitle

\begin{abstract}

For a simple (unbiased) random walk on a connected graph with $n$ vertices, the cover time (the expected number of steps it takes to visit all vertices) is at most $O(n^3)$. We consider locally biased random walks, in which the probability of traversing an edge depends on the degrees of its endpoints. We confirm a conjecture of Abdullah, Cooper and Draief [2015] that the min-degree local bias rule ensures a cover time of $O(n^2)$. For this we formulate and prove the following lemma about spanning trees.

Let $R(e)$ denote for edge $e$ the minimum degree among its two endpoints. We say that a weight function $W$ for the edges is feasible if it is nonnegative, dominated by $R$ (for every edge $W(e) \le R(e)$) and the sum over all edges of the ratios $W(e)/R(e)$ equals $n-1$. For example, in trees $W(e) = R(e)$, and in regular graphs the sum of edge weights is $d(n-1)$.

{\bf Lemma:} for every feasible $W$, the minimum weight spanning tree has total weight $O(n)$.

For regular graphs, a similar lemma was proved by Kahn, Linial, Nisan and Saks [1989].
\end{abstract}
\newpage


\section{Introduction}

Let $G=G(V,E)$ be a simple connected graph with $n$ vertices and $m$ edges.
For any vertex $v\in V$, $d(v)$ denotes the degree of $v$ (the number
of edges incident with $v$), and $N(v)$ denotes the set of neighbors of $v$ (those vertices $u$ for which $(u,v) \in E$). Let $c : E \rightarrow R^+$ be a function, referred to as {\em conductance}, that assigns nonnegative weights to edges of $G$, subject to the condition that the subgraph induced by edges of strictly positive conductances spans all of $V$ and is connected. For every edge $e$ with positive conductance, we refer to $\frac{1}{c(e)}$ as its resistance, and denote it by $r(e)$.
Given $G$ and $c$, we consider the discrete time Markov chain, which we will refer to also as a random walk, whose states are the vertices of $G$, and at each step the random walk moves from the current vertex (say, $v$) to a neighboring vertex (say $u$), chosen at random with probability proportional to $c(v,u)$ (we slightly abuse notation and use $c(u,v)$ to denote $c((v,u))$). Namely, if the chain is at vertex $v$, it moves to each vertex $u \in N(v)$ with probability $\frac{c(v,u)}{\sum_{w \in N(v)} c(v,w)}$.
Such Markov chains are {\em reversible} (see Section~\ref{sec:preliminaries}). In the special case in which $c(e) = r(e) = 1$ for every edge we refer to the resulting Markov chain as a {\em simple random walk}. Given a graph $G(V,E)$ and a conductance function $c$, we shall be interested in the following properties of random walks:

\begin{itemize}

\item
For two vertices $u,v \in V$, the {\em hitting time} $H[u,v]$ is the
expected number of steps it takes a walk that starts at $u$ to reach $v$. The term {\em maximum hitting time} refers to $\max_{u,v \in V}[H[u,v]]$.

\item For two vertices $u,v \in V$, the {\em commute time} $C[u,v]$ is
the expected number of steps that it takes
a walk to go from $u$ to $v$ and back to $u$
(that is, $C[u,v] = H[u,v] + H[v,u]$). The term {\em maximum commute time} refers to $\max_{u,v \in V}[C[u,v]]$, and it cannot exceed twice the maximum hitting time.

\item The {\em cover time} $COV[G,c]$ (or $COV[G]$ for simple random walks)
is the expected number of steps it takes
a random walk to visit all vertices of the graph, starting at the
worst possible vertex (that maximizes this value).

\item The {\em cyclic cover time}, $CYC[G,c]$, is the expected number of
steps it takes a random walk to
visit all vertices of the graph in a prespecified cyclic order, for the
best cyclic order (that minimizes this value). That is,

$$CYC[G,c] =
H[v_{i_1},v_{i_2}] + H[v_{i_2},v_{i_3}] + ... + H[v_{i_{n-1}},v_{i_n}] +
H[v_{i_n},v_{i_1}]$$

where $(i_1,i_2, ... , i_n)$ is a
permutation that minimizes the above sum.

\end{itemize}

Clearly, $CYC[G,c] > COV[G,c]$ and $CYC[G,c] \ge \max_{u,v\in V}[C[u,v]]$.

Abdullah, Cooper and Draief~\cite{ACD} proposed the conductance function $c(u,v) = \frac{1}{\min[d(u),d(v)]}$, which can be equivalently described as $r(u,v) = \min[d(u),d(v)]$. They referred to it as the {\em minimum degree weighting scheme}. For this conductance function, they proved that for every connected graph the maximum hitting time is at most $6n^2$, and concluded from this (using the relation $COV[G] \le \max_{u,v}[H[u,v]]\ln n$, proved by Matthews~\cite{matthews}) that the cover time is at most $O(n^2\log n)$. They further conjectured that with the minimum degree weighting scheme every connected graph has cover time $O(n^2)$. Our main result is a proof of this conjecture, and in fact a stronger result showing that also the cyclic cover time is upper bounded by $O(n^2)$. For cyclic cover time, this result is best possible (up to constant multiplicative factors), because it can be shown that for every reversible Markov chain the cyclic cover time is at least $\Omega(n^2)$.

\subsection{Related work}

For simple random walks, the range of possible values of the cover time is well understood. Aleliunas \etal{} \cite{AKLLR}
showed that for any connected graph, $COV[G] < 2nm$. That work also introduced the spanning tree argument that is used also in establishing other upper bounds cited below, and in fact provides upper bounds on $CYC[G]$ and not just on $COV[G]$.
%
%
For regular graphs the upper bound can be improved to $O(n^2)$, as shown by Kahn \etal{} in~\cite{KLNS}. A more refined connection between regularity and cyclic cover time was provided by Coppersmith \etal{}~\cite{CFS} who proved that for any connected graph $G$,


\begin{equation}\label{eq:CFS}
\frac{3}{10}CYC[G] \le \left(\sum_{v\in V} d_v\right)\left(\sum_{v\in V} {1\over d(v) + 1}\right) \le
CYC[G].
\end{equation}

Observe that for every graph
$$\Omega(n^2) \le \left(\sum_{v\in V} d(v)\right)\left(\sum_{v\in V} {1\over d(v) + 1}\right) \le O(n^3)$$
\noindent and for $d$-regular graphs the value of this expression is $\frac{d}{d+1}n^2$.

Ikeda~\etal{} \cite{IKY} initiated the following line of research: is there a local rule for constructing a conductance function that ensures that for every connected graph the cover time will be $O(n^2)$. As shown in~\cite{IKY}, without further restriction on the class of graphs, $\Omega(n^2)$ is the best one can hope for, e.g., for a path of length $n$. By a local rule one means here that the conductance of an edge $(u,v)$ is a function only of $d(u)$ and $d(v)$. (A nonlocal rule can pick a spanning tree in $G$, give all its edges conductance~1 and all other edges conductance~0. The cover time will then necessarily be $O(n^2)$, by~\cite{AKLLR}.) Ikeda~\etal{} proposed the conductance function $c(u,v) = \frac{1}{\sqrt{d(u)d(v)}}$, showed that it ensures that the maximum hitting time is $O(n^2)$, concluded (using~\cite{matthews}) that the cover time is $O(n^2 \log n)$, but left open the question of whether there is any local rule that ensures cover time of $O(n^2)$. Abdullah~\etal{} \cite{ACD} proposed the conductance function $c(u,v) = \frac{1}{\min[d(u),d(v)]}$, proved for it bounds similar to those proved in~\cite{IKY}, and explicitly conjectured that it leads to a cover time of $O(n^2)$ (a conjecture that we confirm in this paper). (There are additional results in~\cite{ACD} that are not directly relevant to the current paper.)

A different but related approach for obtaining Markov chains with maximum hitting times at most $O(n^2)$, that of so called {\em Metropolis walks}, was proposed by Nonaka \etal{} \cite{NOSY}, who also showed that it does not give a cover time better than $\Omega(n^2 \log n)$ on the {\em glitter star} graph (see Figure~\ref{fig:glitter}).

\subsection{Our results}

\subsubsection{Main results}

Our main theorem is the following.

\begin{thm}
\label{thm:upper}
For every connected graph on $n$ vertices, the conductance function implied by the minimum degree weighting scheme of Abdullah~\etal{} \cite{ACD} gives a random walk with cyclic cover time at most $18n^2$.
\end{thm}

Theorem~\ref{thm:upper} is best possible in the following sense:

\begin{proposition}
\label{pro:lower}
For every connected graph on $n$ vertices and every conductance function the associated random walk has cyclic cover time at least $\frac{1}{2}n^2$.
\end{proposition}

The upper bound of Theorem~\ref{thm:upper} of course applies also to the cover time. However, the lower bound in Proposition~\ref{pro:lower} does not hold for the cover time.

Our proof of Theorem~\ref{thm:upper} is based on a Lemma that can be stated without any reference to random walks. Given a connected graph $G(V,E)$, for every $(u,v) \in E$ define $r(u,v) = \min[d(u),d(v)]$. A weight function $w : E \longrightarrow R$ is said to be {\em feasible} if it satisfies the following two conditions:

\begin{itemize}

\item $0 \le w(e) \le r(e)$ for every $e \in E$ (where $r(u,v) = \min[d(u),d(v)]$).

\item $\sum_{e\in E} \frac{w(e)}{r(e)} = n-1$.

\end{itemize}

\begin{lemma}
\label{lem:main}
For every connected graph $G$ on $n$ vertices and every feasible weight function $w$, graph $G$ has a spanning tree of total edge weight at most $9n$.
\end{lemma}

The proof of Lemma~\ref{lem:main} is the main new technical contribution of our paper.

We did not attempt to optimize the leading constants in the statements of Lemma~\ref{lem:main}, Theorem~\ref{thm:upper} and Proposition~\ref{pro:lower}.

\subsubsection{Additional results}
\label{sec:local}

A {\em local rule} for conductance is one by which the conductance of an edge $(u,v)$ depends on $d(u)$ and $d(v)$.
Two local rules for conductance functions are said to be {\em equivalent} to each other if one can be obtained from the other by scaling (e.g., the rules $c(u,v) = d_u + d_v$ and the $c(u,v) = 2d_u + 2d_v$ are equivalent). Two local rules $c_1$ and $c_2$ for conductance functions are said to be {\em roughly equivalent} if there are constants $0 < \alpha \le \beta$ such that for every edge $e$, $\alpha c_1(e) \le c_2(e) \le \beta c_1(e)$. We provide an additional aspect by which Theorem~\ref{thm:upper} is best possible.

\begin{proposition}
\label{pro:local}
Every local rule for establishing a conductance function is either roughly equivalent to the minimum degree weighting scheme, or there are graphs on which the associated random walk has cyclic cover time
    $\omega(n^2)$ (namely, not bounded by $O(n^2)$).
\end{proposition}

We do not know if cyclic cover time can be replaced by cover time in Proposition~\ref{pro:local}.

%

\section{Preliminaries}
\label{sec:preliminaries}

We provide some background on random walks (for more details, see for example~\cite{aldous}). Given a random walk on an $n$-vertex connected graph $G(V,E)$ with conductance $c$, for every vertex $v\in V$ define
$$\pi(v) = \frac{\sum_{u \in N(v)} c(u,v)}{\sum_{e \in E} c(e)}$$
\noindent and observe that $\sum_{v\in V} \pi(v) = 1$.  For random walks with a stationary distribution (when $G$ is connected and has odd cycles) $\pi$ as defined above coincides with the stationary distribution. The expected time it takes a walk that starts at $v$ to return to $v$ satisfies:

\begin{equation}\label{eq:return}
H[v,v] = \frac{1}{\pi(v)}
\end{equation}

The equality
$$\pi(u)\frac{c(u,v)}{\sum_{w \in N(u)} c(u,w)} = \pi(v)\frac{c(u,v)}{\sum_{w \in N(v)} c(u,w)}$$
\noindent which holds for every edge $(u,v) \in E$ implies that the walk is {\em reversible}. For reversible random walks, the following identity holds for every sequence $v_1, \ldots, v_k$ of vertices:

$$
H[v_k, v_1] + \sum_{i=1}^{k-1} H[v_i,v_{i+1}] = H[v_1, v_k] + \sum_{i=k}^{2} H[v_i,v_{i-1}]
$$

Consequently,

\begin{equation}\label{eq:cyclic}
H[v_k, v_1] + \sum_{i=1}^{k-1} H[v_i,v_{i+1}] = \frac{1}{2}\left(C[v_k, v_1] + \sum_{i=1}^{k-1} C[v_i,v_{i+1}]\right)
\end{equation}

The following approach, initiated by~\cite{AKLLR} (see also~\cite{KLNS,CFS}, among other works based on this approach), can be used in order to upper bound the cover time, and in fact also the cyclic cover time.
Given an undirected connected graph $G = G(V,E)$ and a conductance function $c$, consider an arbitrary spanning tree $T$ in $G$. Then the cyclic cover time is upper bounded by the sum of commute times along the edges of $T$. Namely (here $T$ is thought of as the set of edges that make up the spanning tree):

\begin{equation}\label{eq:ST}
CYC(G,c) \le \sum_{(u,v) \in T} C[u,v]
\end{equation}

To get a handle on commute times,
it will be convenient for us to use the well known correspondence between random
walks and resistance of electrical networks. We state here without proof the properties that we shall use, and the reader is referred to~\cite{DS,CRRST,tetali} for further details.

Given a conductance function $c$, each edge of $G(V,E)$ is viewed as a resistor of resistance $r(e) = \frac{1}{c(e)}$ ohm. The
{\em effective resistance} between vertices $u$ and $v$, denoted by $R(u,v)$, is the
voltage that develops at $u$ if a current of 1 amp is injected into
$u$, and $v$ is grounded. The effective resistance
captures the commute time in the sense that for every two vertices $u$ and $v$,

\begin{equation}\label{eq:commute}
C[u,v] = 2R(u,v)\sum_{e\in E} c(e)
\end{equation}

Combining inequality~(\ref{eq:ST}) with equality~(\ref{eq:commute}) he have the following theorem (known and used by previous work):

\begin{thm}
\label{thm:ST}
Let $G(V,E)$ be an arbitrary connected graph, let $c$ by a conductance function, and for every $e \in E$ let $R(e)$ denote the induced effective resistance. Then given any collection $T$ of edges that form a spanning tree in $G$, the cyclic cover time satisfies:

$$CYC[G,c] \le 2\left(\sum_{e\in E} c(e)\right) \cdot \left(\sum_{e \in T} R(e)\right)$$
\end{thm}

To make use of Theorem~\ref{thm:ST} we need to use properties of effective resistance.
We present without proof an identity due to Foster \cite{foster,tetali}:

\begin{lemma}
\label{lem:foster}
Let $G(V,E)$ be a connected graph with $n$ vertices, and for every edge $e \in E$ let $r(e) > 0$ denote its resistance (inverse of the conductance function $c$ in our usage).
Then the resistances and effective resistances along the edges of $G$ satisfy the following identity:

$$\sum_{e\in E} \frac{R(e)}{r(e)} \; = \; n - 1$$

\end{lemma}

In this paper we shall use only two properties of effective resistance, listed below:

\begin{enumerate}

\item Allowed range of values: $0 \le R(e) \le r(e)$ for every edge $e \in E$.

\item Foster's identity: $\sum_{e\in E} \frac{R(e)}{r(e)} \; = \; n - 1$.

\end{enumerate}

We remark that the above two properties suffice in order to prove the upper bounds on the cover time provided in~\cite{AKLLR,KLNS}, but the proofs of the bounds in~\cite{CFS} (see Inequality~(\ref{eq:CFS})) use additional properties not listed here.

\subsection{The minimum degree weighting scheme}

Recall that in the minimum degree weighting scheme $r(u,v) = \min[d(u),d(v)]$, or equivalently, $c(u,v) = \frac{1}{\min[d(u),d(v)]}$. The following proposition is proved in~\cite{ACD} (and a similar proposition is proved in~\cite{IKY} for a related local rule). We repeat its proof for completeness.

\begin{proposition}
\label{pro:ACD}
Let $G(V,E)$ be an arbitrary $n$ vertex graph. Then for the conductance function $c(u,v) = \frac{1}{\min[d(u),d(v)]}$ it holds that
$$\sum_{e\in E} c(e) \le n-1$$
\end{proposition}

\begin{proof}
Order the vertices of $G$ from~1 to~$n$ in order of increasing degrees (breaking ties arbitrarily). Then:

$$\sum_{e\in E} c(e) = \sum_{i\in V} \sum_{j \in N(i) ; j > i} c(u,v) =  \sum_{i\in V} \sum_{j \in N(i) ; j > i} \frac{1}{d(i)} \le \sum_{i=1}^{n-1} d(i)\cdot \frac{1}{d(i)} \le n-1$$
\end{proof}

The star graph is an example for which equality holds in Proposition~\ref{pro:ACD}.

\section{Proof of Theorem~\ref{thm:upper}}
\label{sec:mainproof}

In this section we prove our main theorem, namely, Theorem~\ref{thm:upper}.
We first prove our key lemma, namely, Lemma~\ref{lem:main}. We remark that previous work~\cite{KLNS} can be shown to imply that Lemma~\ref{lem:main} holds in the special case in which $G$ is regular, but the proof techniques used there do not seem to suffice in order to prove Lemma~\ref{lem:main} in its full generality.

Before proving Lemma~\ref{lem:main}, it is instructive to consider another special of case of this lemma, namely when the graph $G$ is a tree.  In this case $G$ has a unique spanning tree, and moreover, $w(e) = r(e)$ is the unique feasible weight function.

\begin{proposition}
\label{pro:tree}
For every graph $G$ that is a tree and for its unique feasible weight function $w$, the sum of edge weights at most $2n-4$.
\end{proposition}

\begin{proof}
Recall that the unique feasible weight function for the tree is $w(u,v) = \min[d(u),d(v)]$ for every edge $(u,v)$. Let $r$ be the highest degree vertex in $G$ (breaking ties arbitrarily) and orient all edges away from $r$. For every oriented edge $(u,v)$ we have that $w(u,v) \le d_v$. As every tree vertex except for $r$ has exactly one edge oriented into it we have that $\sum_{(u,v)\in E} w(u,v) \le \sum_{v \in V \setminus \{r\}} d(v) = 2(n-1) - d_r \le 2n - 4$.
\end{proof}

The proof of Lemma~\ref{lem:main} (for general graphs) is considerably more difficult than the proof of Proposition~\ref{pro:tree} (for trees). We remark that one of the steps of the proof of Lemma~\ref{lem:main} (namely, Step~3) is based on principles similar to those used in the proof of Proposition~\ref{pro:tree}.

We now prove Lemma~\ref{lem:main}.

\begin{proof}
Denote by $d_{G}\left(v\right)$ the degree of a vertex $v$ in $G$.
Throughout the proof, we consider a directed graph $H$  with the same vertex
set as $G$.
Denote  by $d_{H}^{in}\left(v\right)$ the the number of incoming
edges
to $v$ in $H$.

Let $r\left(u,v\right)$ be the minimum between $d_{G}\left(v\right)$
and $d_{G}\left(u\right)$. Recall that for every edge $\left(u,v\right)$
it holds that $0\le w\left(u,v\right)\le r\left(u,v\right)$. For
every edge $e$, we set $\rho\left(e\right)=\frac{w\left(e\right)}{r\left(e\right)}$.
Recall that
\[
\sum_{e\in E}\rho\left(e\right)=n-1\,.
\]

We construct a spanning tree $T$ in three steps.
\medskip

\noindent {\bf Step 1: constructing a directed forest.}
Define the weight of a vertex $v$, $\rho(v)$, to be $\sum_{u\in N(v)}\rho(u,v)$
and note that
\begin{equation}
\sum_{v\in G}\rho(v)=2(n-1)\,. \label{eq:rho}
\end{equation}
Let $\alpha$ be a positive constant smaller than $1$ whose value will be determined later. By averaging, for every
vertex $v\in V$ it holds that at least $\lceil d_{G}\left(v\right) (1-\alpha) \rceil$
of its incident edges satisfies
\begin{equation}
\rho\left(u,v\right)\le \frac{1}{\alpha}\frac{\rho(v)}{d_{G}\left(v\right)}\,.\label{eq:sort}
\end{equation}
We refer to these edges as {\em good} edges. (Observe that an edge is defined as good based on having a low ratio $\rho(e)$, whereas it might appear more natural to base this on having low weight $w(e)$. The reason for our choice of definition will become apparent in the proof of Claim~\ref{claim:size}.)

Let $H$ be the subgraph of $G$ induced by the good edges. Note that the graph
$H$ may not be connected. Orient every good edge in $H$ towards
the vertex that with respect to which this edge is good (some edges might
be bidirectional). Note that
\begin{equation}\label{eq:in}
d_{H}^{in}\left(v\right)\ge (1-\alpha)d_{G}\left(v\right).
\end{equation}
Note also that the definition of $\rho(e)$ ($=\frac{w(e)}{r(e)}$) together with the fact that $r(u,v) \le d_G(v)$ and with Inequality~(\ref{eq:sort}) imply that for every edge $(u,v)$ incoming into $v$ in $H$ it holds that
\begin{equation}\label{eq:w}
w(u,v) \le \frac{1}{\alpha} \rho(v).
\end{equation}

A  path $u_1,u_2...,u_k$ from $u_1$ to $u_k$ is said to be directed path if the edge $u_i,u_{i+1}$ is directed to $u_{i+1}$ for all $1\le i\le k-1$.
Define a rooted directed tree  as a tree with a directed path from the root to any other  vertex in the tree.
A directed spanning forest is a set of disjoint rooted directed trees that spans  the graph.

Consider a directed spanning forest $F$ of $H$ obtained in a greedy
manner as follows. Initially, mark all vertices as {\em uncovered}. Now proceed iteratively until all vertices are marked as {\em covered}. In a single iteration, start from the uncovered vertex $v$ with highest $\rho (v)$ (breaking ties arbitrarily) and construct a maximal directed tree with $v$ as its root. (Given $v$, there might be several possible maximal directed trees to chose from, though they all contain the same set of vertices and differ only by their sets of edges.) Mark all vertices of the tree as {\em covered}, and remove them (and their incident edges) from the graph. Let $k$ denote the total number of iterations until all vertices are covered.
Hence $F$ contains $k$ directed trees, and we denote them by $T_{1},T_{2},...,T_k$ according to the order by which they were generated.

Denote by $Root\left(T_{i}\right)$ the root of $T_{i}$ and by $Roots$ the set of roots of $T_{1},T_{2},...,T_k$.
\begin{claim}
\label{lem:greedy cost}
The sum of weights of all the edges in $F$ is at most $\sum_{v\in G\setminus Roots}\frac{1}{\alpha}\rho(v)$.
\end{claim}
\begin{proof}
Consider a tree $T_{i}$ in
$F$. For every non-root vertex $v$ in $T_{i}$, the weight of the edge connecting
it to its parent is at most $\frac{1}{\alpha} \rho\left(v\right)$, by Inequality~(\ref{eq:w}).
In total the cost of $T_{i}$ is at most $\sum_{v\in T_{i}\setminus Root\left(T_{i}\right)}\frac{1}{\alpha}\rho(v)$ and
the cost of $F$ is at most $\sum_{\left\{T_{i}\right\}}\sum_{v\in T_{i}\setminus Root\left(T_{i}\right)}\frac{1}{\alpha}\rho(v)=\sum_{v\in G\setminus Roots}\frac{1}{\alpha}\rho(v)$.
\end{proof}

We add all the edges of $F$ to $T$.
\medskip

\noindent {\bf Step 2: enforcing a size requirement.}
Now we add edges to $F$ to obtain $F^{*}$, which is also a spanning forest of $H$ (though the trees in this forest need not be directed -- a tree might have multiple sources and vertices in the tree might have multiple incoming edges). The goal of this step is to satisfy the following
{\em size requirement}. In $F^{*}$,
for every tree $T_{i}$ and for every vertex $v$ in $T_{i}$
we have that
the number of vertices in $T_{i}$ (denote this quantity by $\left|T_{i}\right|$)
is at least $(1-\alpha)d_{G}\left(v\right)$. All the edges we add
in this step are also added to $T$.


We use the following claim.
\begin{claim}
\label{claim:greedy property}Let $1\le i<j\le k$. There are
no edges in $H$ that are directed from $T_{i}$ to $T_{j}$.
\end{claim}
\begin{proof}
The claim follows by the way we constructed $F$. Assume towards a contradiction that there exists a directed edge (from $u$ to $v$)  from $T_i$ to $T_j$ for some $i < j$. It follows that when $T_i$ was constructed we could add $v$ to $T_i$. This contradicts the maximality of $T_i$.
\end{proof}

Claim~\ref{claim:greedy property} implies that for every
vertex in $T_{k}$, all of its incoming edges in $H$ have their other endpoint
in $T_{k}$. Inequality~(\ref{eq:in}) then implies that $T_{k}$ satisfies
the size requirement. We now proceed by induction to enforce the size requirement on the remaining directed trees. In the process we shall add edges to $F$ thus connecting disjoint directed trees into new trees.

In the inductive step, the size requirement for $T_{m+1},...,T_{k}$ (or more exactly, for the trees that resulted from processing $T_{m+1},...,T_{k}$) is assumed to hold.
Consider $T_{m}$,
and let $v$ be the vertex in $T_m$ with highest $d_{G}\left(v\right)$.

If $(1-\alpha)d_{G}(v)\le \left|T_m\right|$ then the size requirement holds and nothing needs to be done. Hence we may assume that $(1-\alpha)d_{G}(v)> \left|T_m\right|$. As $d_H^{in}(V) \ge (1-\alpha)d_{G}(v)$ it follows that $v$ has incoming edges from outside of $T_m$. By Claim~\ref{claim:greedy property} these incoming edges come from trees that were already processed in previous inductive steps.
Suppose that $v$ has an incoming edge from a vertex $u$ for which $d_G(u) \ge  d_G(v)$.
Then add the edge $\left(u,v\right)$
to $F^{*}$, paying at most $\alpha\rho\left(v\right)$ (by Inequality~(\ref{eq:w})) and satisfying
the size requirement.

It remains to consider the case that every incoming edge to $v$ from
$T_{m+1},...,T_{k_{1}}$ comes from vertices with degree smaller than
$d_G(v)$. Iteratively ($i$ goes from $1$ to at most $(1-\alpha)d_{G}\left(v\right) $), we add
edges $\left(u_{i},v\right)$ for $u_{i}$ in some $T_{n_{i}}$ (note
that $n_{i}>m$ and that by the assumption $\left|T_{n_{i}}\right|$
is at least $(1-\alpha)d_{G}\left(u_{i}\right)$ ) until $T_m$ combined with the trees added
to it  has $(1-\alpha)d_{G}\left(v\right) $ vertices.
(The process must end as the set of candidate trees that can be added to $T_m$ includes all end points of incoming edges to $v$, and  $d_H^{in}(v) \ge (1-\alpha) d_{G}\left(v\right)$.)

Considering all the iterations except for the last one it holds that
\[
\sum_{i=1}^{\text{\# iterations}-1}(1-\alpha)d_{G}\left(u_{i}\right)\le\sum_{i=1}^{\text{\# iterations}-1}\left|T_{n_{i}}\right|\le (1-\alpha) d_{G}\left(v\right)\,.
\]
and considering also the last iteration that adds an edge connected to a vertex of degree at most $d_G(v)$ we have:
\begin{equation}
\sum_{i=1}^{\text{\# iterations}}d_{G}\left(u_{i}\right)\le 2d_{G}\left(v\right)\,.\label{eq:2}
\end{equation}

\begin{claim}
\label{claim:size}
The sum weight of all edges added in all iterations satisfies
$$\sum_{i=1}^{\text{\# iterations}}w\left(u_{i},v\right) \le \frac{2}{\alpha}\rho\left(v\right).$$
\end{claim}

\begin{proof}
The proof can be derived as follows:
\begin{align*}\sum_{i=1}^{\text{\# iterations}}w\left(u_{i},v\right) & =\sum_{i=1}^{\text{\# iterations}}d_{G}\left(u_{i}\right)\frac{w\left(u_{i},v\right)}{d_{G}\left(u_{i}\right)}\\
 & =\sum_{i=1}^{\text{\# iterations}}d_{G}\left(u_{i}\right)\frac{w\left(u_{i},v\right)}{r\left(u_{i},v\right)}\\
 & =\sum_{i=1}^{\text{\# iterations}}d_{G}\left(u_{i}\right)\rho\left(u_{i},v\right)\\
 & \le\sum_{i=1}^{\text{\# iterations}}d_{G}\left(u_{i}\right)\frac{1}{\alpha}\frac{\rho\left(v\right)}{d_{G}\left(v\right)}\\
 & \le 2\frac{1}{\alpha}\rho\left(v\right)\,.
\end{align*}
The two last inequalities follow by Inequality~(\ref{eq:sort}) and
Inequality~(\ref{eq:2}).
\end{proof}

To conclude, we started with the trees $T_{1},T_{2},...,T_k$. For each tree $T_{i}$ the total weight of the added edges (to enforce the size requirement) is at most $\frac{2}{\alpha}\rho (v)$, for some $v\in T_i$. By the greedy choice of root in Step~1 we have that $\rho (v)\le \rho (Root (T_i))$.
In total the cost of edges that we added to $T$
to obtain $F^{*}$ is at most $\sum_{v\in Roots} \frac{2}{\alpha} \rho (v)$.
\medskip

\noindent {\bf Step 3: connecting the forest into a tree.}
Now we use additional edges from $G$ (regardless of whether they are good edges) to convert $F^{*}$ into a spanning
tree of $G$. Again, all the edges
that we add are added to $T$ as well. Let $t$ denote the number of trees
in $F^{*}$. Let $E_{2} \subset E$ be an arbitrary set of $t-1$ edges  such that
adding them to $F^{*}$ results in a spanning tree to $G$ (such a set $E_2$ must exist because $G$ is connected). Let $W$
be a subgraph of $G$ induced by
$F^{*} \cup E_{2}$ ($W$ is a spanning tree of $G$). Consider the tree $T'$ obtained by
contracting (in $W$) all the vertices in each connected component
in $H$ into a single vertex (ignore edge orientation). In other words,
each vertex in $T'$ represents a connected component in $F^{*}$
and the cost of $E_{2}$ is the cost of $T'$.

For each tree $F_{i}$ in $F^{*}$ we denote by $d\left(F_{i}\right)$
the quantity $max_{v\in F_{i}}d_{G}\left(v\right)$. Note that by
the size requirement it holds that
\begin{equation}
\left|F_{i}\right|\ge(1-\alpha)d\left(F_{i}\right)\,.\label{eq:1-1}
\end{equation}
 We define a root for  $T'$  at some arbitrary vertex/component and
each vertex/component $F_{i}$ pays for the edge (denoted by $e\left(F_{i}\right)$) connecting it to
its parent. The cost of $T'$ is
\begin{align}
\sum_{i=1}^{t-1}e\left(F_{i}\right) & \le\sum_{i=1}^{t-1}d\left(F_{i}\right)\nonumber \\
 & \le\sum_{i=1}^{t-1}\frac{1}{(1-\alpha)}\left|F_{i}\right|\label{eq:2-1}\\
 & \le\frac{1}{(1-\alpha)}\left(n-1\right)\,,\nonumber
\end{align}
where Inequality~\ref{eq:2-1} follows by Inequality~\ref{eq:1-1} .
\medskip

\noindent {\bf Overall cost.}
Steps 1, 2 and 3 uses edges of total weight at most $\sum_{v\in G\setminus Roots}\frac{1}{\alpha}\rho(v)$,
 $\sum_{v\in Roots} 2\frac{1}{\alpha} \rho (v)$ and $\frac{1}{(1-\alpha)}\left(n-1\right)$, respectively. Hence the total weight of edges used in $T$ is
 \begin{align*}
 \sum_{v\in G\setminus Roots}\frac{1}{\alpha} \rho(v)+\sum_{v\in Roots} 2\frac{1}{\alpha} \rho (v)+\frac{1}{(1-\alpha)}\left(n-1\right) & \le \sum_{v\in G}2\frac{1}{\alpha}\rho (v)+\frac{1}{(1-\alpha)}\left(n-1\right) \\
  & \le \left(4\frac{1}{\alpha}+\frac{1}{(1-\alpha)}\right)\left(n-1\right)\,,
 \end{align*}
 where the last inequality  follows by Inequality~(\ref{eq:rho}). By setting $\alpha=\frac{2}{3}$,  the proof follows.
\end{proof}

We now turn to Theorem~\ref{thm:upper}. Our first step is to observe that the
combination of Theorem~\ref{thm:ST} and Proposition~\ref{pro:ACD} implies the following corollary:

\begin{cor}
\label{cor:ST}
Let $G(V,E)$ be an arbitrary $n$ vertex connected graph, equipped with the minimum degree conductance function $c(u,v) = \frac{1}{\min[d(u),d(v)]}$.
For every $e \in E$ let $R(e)$ denote the induced effective resistance. Then given any collection $T$ of edges that form a spanning tree in $G$, the cyclic cover time satisfies:

$$CYC[G,c] \le 2(n-1) \left(\sum_{e \in T} R(e)\right)$$
\end{cor}

Equipped with Lemma~\ref{lem:main} and with Corollary~\ref{cor:ST}, we now prove Theorem~\ref{thm:upper}.

\begin{proof}
Lemma~\ref{lem:main} shows that for every feasible weight function there is a tree $T$ for which the sum of edge weights is at most $9n$. The effective resistance function $R$ is a feasible weight function, and hence there is a tree $T$ for which the sum of effective resistances of its edges is at most $9n$. Combining this with Corollary~\ref{cor:ST} we obtain:

$$CYC[G,c] \le 2(n-1) \left(\sum_{e \in T} R(e)\right) < 2n \cdot 9n = 18n.$$
\end{proof}

We remark that if $G$ happens to be a tree, then replacing Lemma~\ref{lem:main} by Proposition~\ref{pro:tree} in the proof of Theorem~\ref{thm:upper} shows that $CYC[G,c] < 4n$.

\section{Optimality of Theorem~\ref{thm:upper}}
\label{sec:optimality}

\subsection{Proof of Proposition~\ref{pro:lower}}

We now prove Proposition~\ref{pro:lower}.

\begin{proof}
Let $G(V,E)$ be an arbitrary connected graph on $n$ vertices equipped with an arbitrary conductance function $c$. W.l.o.g., let the cyclic order that minimizes the cyclic cover time be $v_1, v_2, \ldots, v_n, v_1$. Then

$$CYC[G,c] = H[v_n,v_1] + \sum_{i=1}^{n-1} H[v_i,v_{i+1}] = \frac{1}{2}\left(C[v_n, v_1] + \sum_{i=1}^{n-1} C[v_i,v_{i+1}]\right)$$
where the first equality is by definition and the second equality follows from Equation~(\ref{eq:cyclic}).

Observe that for any two vertices $u$ and $v$, $H[u,u] \le C[u,v]$, and likewise, $H[v,v] \le C[u,v]$. Consequently $C[u,v] \ge \frac{1}{2}\left(H[u,u] + H[v,v]\right)$. Using Equation~(\ref{eq:return}) we conclude that

$$CYC[G,c] \ge \frac{1}{2}\sum_{v\in V} \frac{1}{\pi(v)} \ge \frac{1}{2}n^2$$
where the last inequality follows from straightforward convexity arguments, using positivity of $\pi(v)$ and the equality $\sum_{v \in V} \pi(v) = 1$.
\end{proof}

The lower bound in Proposition~\ref{pro:lower} is best possible when $n=2$, but not for larger $n$. For the special case of simple random walks, a stronger lower bound is provided by Inequality~(\ref{eq:CFS}), and that inequality is tight for simple random walks on complete graphs (there the cyclic cover time is $n(n-1)$.

\subsection{Local rules in general}

This section contains (among other things) claims regarding the cover times for various specific graphs. In all cases, proofs of these claims are relatively straightforward, and hence omitted for brevity.

Recall that for simple random walks on regular graphs the cover time is $O(n^2)$, but might be $\Omega(n^3)$ for nonregular graphs. The Lollipop graph, composed of a clique of size $2n/3$ joined to a path of length $n/3$ (see Figure~\ref{fig:lollipop}) has cover time  of roughly $\frac{4n^3}{27}$ (for a walk starting at a clique vertex), and this is highest possible (up to low order terms), as shown in~\cite{feige1}.

\begin{figure}

	\begin{centering}
		\includegraphics[scale=0.6]{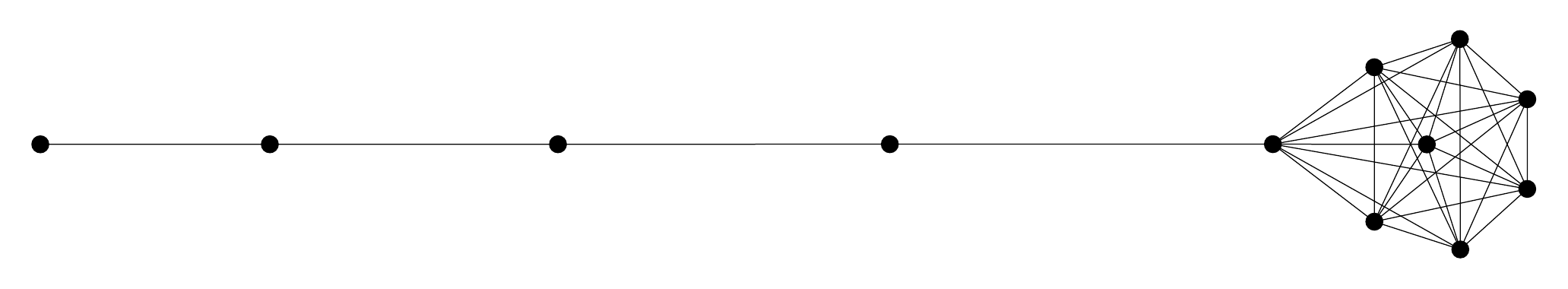}\protect
		\caption{Lollipop graph}
		\label{fig:lollipop}
		\par\end{centering}
	
\end{figure}

In this section we discuss various candidate approaches for modifying simple random walks in order to obtain random walks whose cover time is $O(n^2)$ for every graph, whether regular or not.  As proved in~\cite{IKY}, reversible random walks on the path have cover time at least $\Omega(n^2)$ regardless of the conductance function. Hence bounds of the form $O(n^2)$ are the best we can hope for (if they they are expressed as a function of $n$ and need to hold for every graph).

\begin{rem}
It is desirable to obtain the even stronger property of having cyclic cover time of $O(n^2)$. Such a bound has the following interpretation. Recall that cyclic cover time is tightly related (via Inequality~\ref{eq:CFS}) to natural measures of regularity of a graph. In this respect, achieving $O(n^2)$ cyclic cover time can be thought of as accomplishing the goal of making the underlying graph (nearly) regular (in the eyes of random walks).

An alternative feature unique to regular graphs is having a uniform stationary distribution. However, no modification of random walks will enforce a nearly stationary distribution for graphs such as the star graph.
\end{rem}

An advantage of simple random walks is that they can be implemented with very little local knowledge -- in considering where to move next, all one needs to know is who the neighbors of the current vertex are. In contrast, local rules such as those considered in this paper require significantly more local knowledge -- one needs to know also the degrees of the neighbors. One may hope that there are lighter modifications to simple random walks that ensure an $O(n^2)$ cover time (as good as that given by the minimum degree local rule, but easier to implement). One candidate modification is to allow the walk (say, at a vertex $v$) to remember the vertex (say, $u$) it last came from, and the degree of $u$. The next step is then uniform over all neighbors of $v$ except for $u$, and the probability of going to $u$ might differ (might be lower or higher). Using such a modification one can implement (for example) non-backtracking random walks~\cite{ABLS} and Metropolis random walks~\cite{NOSY}. However, for the clique-star graph that is composed of a clique of size $n/2$ connected via a matching to $n/2$ independent vertices (see Figure~\ref{fig:split}), every such random walk has cover time $\Omega(n^2\log n)$.

\begin{figure}

	\begin{centering}
		\includegraphics[scale=0.6]{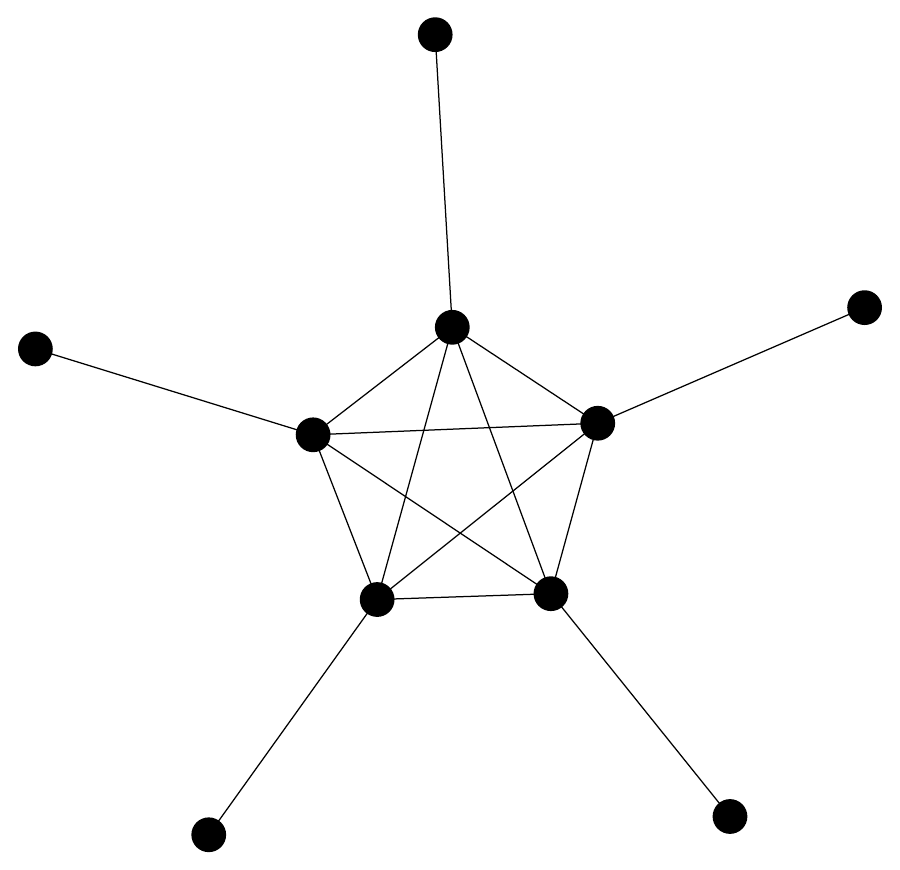}\protect
		\caption{Clique-star graph}
		\label{fig:split}
		\par\end{centering}
	
\end{figure}

Let us return to local rules (as in Section~\ref{sec:local}). We have seen that the minimum degree local rule ensures not just cover time but also cyclic cover time of $O(n^2)$. In this last respect it is unique, up to rough equivalence, as stated in Proposition~\ref{pro:local} that we prove next.

\begin{proof}
Let $c'$ be an arbitrary conductance function not roughly equivalent to the minimum degree one. Namely, for every $\beta > 1$ (and scaling $c'$ by a fixed constant as needed) there are degrees $d_1 \le d_3$ and $d_2 \le d_4$ such that $c'(u,v) = \frac{1}{d_1}$ if $d(u) = d_1$ and $d(v) = d_3$ (referred to as edges of type~1), and $c'(u,v) \le \frac{1}{\beta d_2}$ if $d(u) = d_2$ and $d(v) = d_4$ (referred to as edges of type~2). Moreover $d_3, d_4 \ge 2$ (because when $n > 2$ there are no edges with both endpoints of degree~1).

Consider a graph $G$ composed of $d_4 - 1$ copies of complete bipartite graphs $K_{d_3 - 1, d_1}$ and $d_3 - 1$ copies of complete bipartite graphs $K_{d_4 - 1, d_2}$, and make the graph connected by adding a perfect matching on the right hand side vertices of these bipartite graphs. (If the number of right hand side vertices is odd, make two copies of the above construction and then perfect matchings exist.) Observe that $G$ has exactly $(d_3 - 1)(d_4 - 1)$ left hand side vertices of degree $d_1$ (referred to as vertices of type~1) and all their incident edges are of type~1, and exactly $(d_3 - 1)(d_4 - 1)$ left hand side vertices of degree $d_3$ (referred to as vertices of type~2) and all there incident edges are of type~2. Moreover, the total number of vertices in $G$ is $n = 2(d_3 - 1)(d_4 - 1) + (d_4 - 1)d_1 + (d_3 - 1)d_2 \le 6(d_3 - 1)(d_4 - 1)$.

Consider now the stationary distribution on $G$. There are at least $n/6$ vertices of type~1, and for each such vertex the sum of conductances of its incident edges is~1. Hence $\sum_{e\in E} c(e) \ge \frac{n}{6}$. There are at least $n/6$ vertices of type~2, and for each such vertex the sum of conductances of its incident edges is $\frac{1}{\beta}$, and hence for such a vertex $v$ we have that $\pi(v) \le \frac{6}{\beta n}$. Using the relation $CYC[G,c] \ge \frac{1}{2}\sum_{v \in V} \frac{1}{\pi(v)}$ we obtain that $CYC[G,c] \ge \frac{n}{6} \cdot \frac{\beta n}{6} = \frac{\beta}{36}n^2$. As $\beta$ cannot be bounded by any fixed constant (because $c'$ is not roughly equivalent to the minimum degree local rule), the cyclic cover time with conductance function $c'$ cannot be bounded by $O(n^2)$.
\end{proof}

Some properties of the minimum weight local rule include the following:

\begin{enumerate}

\item The total conductance of all edges satisfies $\sum_{e\in E} c(e) \le n$.

\item In the stationary distribution, for every vertex $\pi(v) \ge \frac{1}{2n}$.

\item For every two vertices $u$ and $v$ their commute time is bounded by $C[u,v] = O(n^2)$.

\item The conductance of an edge is a function of the degrees of its endpoints. Moreover, for every edge $(u,v)$ its resistance is in the range $\min[d(u),d(v)] \le r(u,v) \le \max[d(u),d(v)]$.

\end{enumerate}

One might hope that the above list of properties suffices to ensure a cover time of $O(n^2)$. However, the {\em glitter star} graph  (considered previously in~\cite{NOSY}) serves as a counter example, see Figure~\ref{fig:glitter}. In this graph there is a central vertex and paths of length two connected to it (see Figure~\ref{fig:glitter}). For the conductance function $c(u,v) = \frac{1}{\max[d_u,d_v]}$ all above properties hold but nevertheless the cover time is $\Omega(n^2\log n)$.

\begin{figure}

	\begin{centering}
		\includegraphics[scale=0.6]{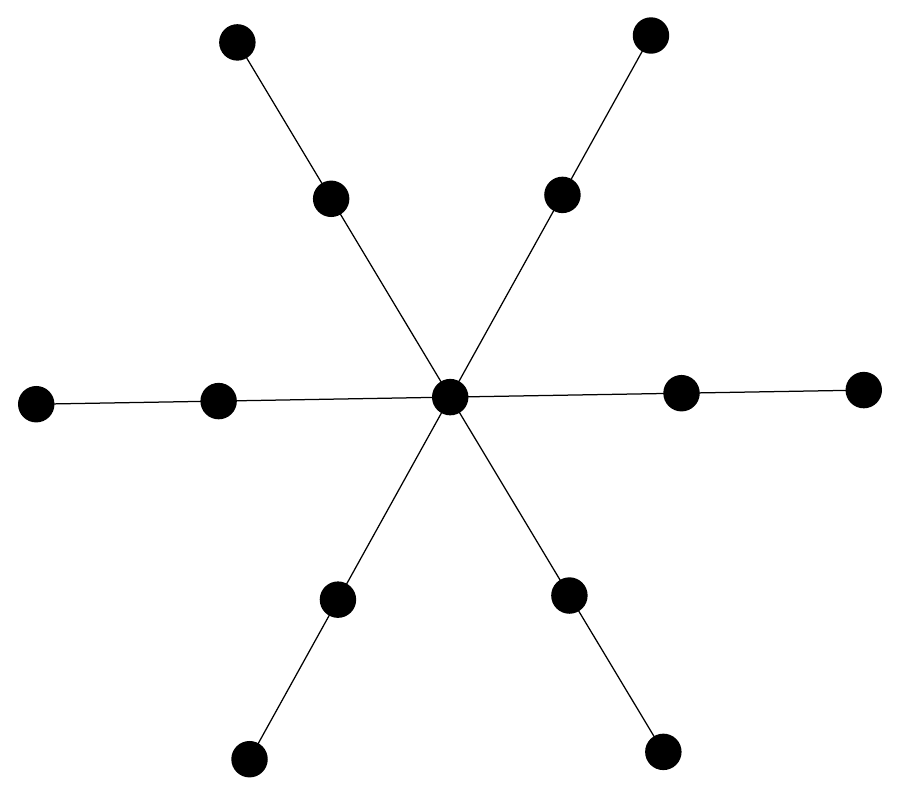}\protect
		\caption{Glitter star graph}
		\label{fig:glitter}
		\par\end{centering}
	
\end{figure}

Finally, it is worth pointing out that even though the minimum degree local rule ensures $O(n^2)$ cover time, for some graphs it increases the cover time. For example, for a graph composed of two stars whose centers are connected by an edge (see Figure~\ref{fig:double}) the cover time of simple random walks is $\Theta(n\log n)$, but with the minimum degree local rule the cover time is $\Theta(n^2)$.

\begin{figure}

	\begin{centering}
		\includegraphics[scale=0.6]{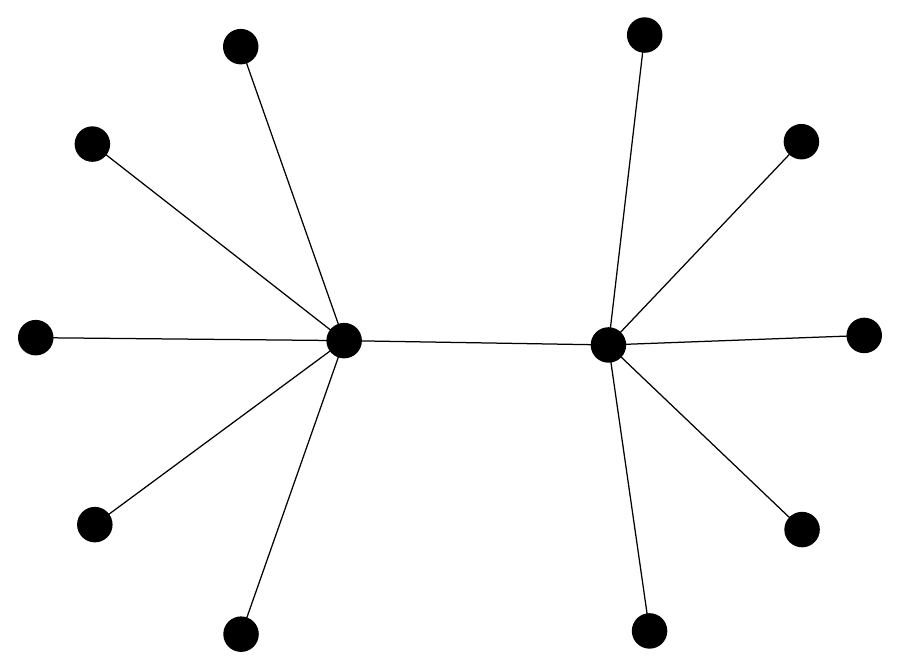}\protect
		\caption{Double star graph}
		\label{fig:double}
		\par\end{centering}
	
\end{figure}

\bigskip
{\bf Acknowledgements:}
Research supported in part by the Israel Science Foundation (grant No. 621/12) and by the I-CORE Program of the Planning and Budgeting Committee and the Israel Science Foundation (grant No. 4/11).

\bibliographystyle{plain}
\bibliography{bib}

\end{document}